\documentclass[letterpaper, 10 pt, conference]{ieeeconf}  

\IEEEoverridecommandlockouts                              

\usepackage{amsmath}
\usepackage{amssymb}
\usepackage{amstext}

\usepackage{graphicx}

\usepackage{algorithm}
\usepackage{algorithmic}

\usepackage{color}

\usepackage{url}

\usepackage{psfrag}

\usepackage[normalem]{ulem}

\usepackage{subcaption}

\usepackage{epstopdf}

\usepackage{booktabs}

\usepackage{tikz}
\usetikzlibrary{arrows}
\usepackage{verbatim}

\usepackage{standalone}

\usepackage{balance}

\usepackage{bm}

\usepackage{wasysym}

\usepackage{euscript}

\usepackage{latexsym}

\newcommand{\mR}{{\mathbb R}}
\newcommand{\mC}{{\mathbb C}}

\newcommand{\Linf}{L_\infty}

\newcommand{\Hinf}{\mathcal{H}_\infty}

\newcommand{\HinfC}{\Hinf(\mC_+)}
\newcommand{\quotientfield}{\mathcal{F}(\Hinf)}

\newcommand{\absfuncwzero}{\phi_{\bar{\tau}}}
\newcommand{\outerfuncwzero}{W_{\bar\tau}}

\newcommand{\real}{\Re}
\newcommand{\imag}{\Im}

\newcommand{\taumax}{\tau_{\rm max}}

\newtheorem{thm}{Theorem}

\newtheorem{lemma}[thm]{Lemma}

\graphicspath{{Figures/}}

\title{\LARGE \bf
Lower bounds on the maximum delay margin by analytic interpolation
}

\author{Axel Ringh$^{1}$, Johan Karlsson$^{1}$, and Anders Lindquist$^{2,1}$
\thanks{*This work was supported by the Swedish Research Council (VR),  grant 2014-5870, the ACCESS Linnaeus Center at KTH, and SJTU-KTH cooperation grant. }
\thanks{$^{1}$Division of Optimization and Systems Theory, Department of Mathematics, KTH Royal Institute of Technology, Stockholm, Sweden. {\tt\small johan.karlsson@math.kth.se}, {\tt\small aringh@kth.se}}%
\thanks{$^{2}$Department of Automation and School of Mathematics, Shanghai
Jiao Tong University, Shanghai, China. {\tt\small alq@math.kth.se}}%
}

\begin{document}

\maketitle
\thispagestyle{empty}
\pagestyle{empty}

\begin{abstract}

We study the delay margin problem in the context of recent works by T. Qi, J. Zhu, and J. Chen, where a sufficient condition for the maximal delay margin  is formulated in terms of an interpolation problem obtained after introducing a rational approximation. Instead we omit the approximation step and solve the same problem directly using techniques from function theory and analytic interpolation. Furthermore, we introduce a constant shift in the domain of the interpolation problem. In this way we are able to improve on their lower bound for the maximum delay margin.



\end{abstract}

\section{Introduction}

Time delays are ubiquitous in linear time invariant (LTI) systems, especially in networks, and may occur through communication delay, computational delay or physical transport delay. Consequently, systems with delay have been the subject of much study in systems and control; see, e.g.,  \cite{gu2003stability, michiels2007stability}, \cite{fridman2014introduction} and references therein.

This paper is devoted to the achievable delay margin in unstable control systems with time delay, a topic that has been studied in various contexts in, e.g., 
\cite{wang1994representation, chen1995new, huang2000robust, kao2007stability, middleton2007achievable, alterman2011robustness, qi2014fundamental, qi2017fundamental, ju2016further, ju2018achievable}. This problem is related to the gain margin and phase margin problems in robust control \cite{doyle1992feedback}, \cite{khargonekar1985non}, but the delay margin problem is  more complicated,  and many unsolved problems remain. Loosely speaking, we are looking for the largest time delay $\taumax$ such that  there exists an LTI controller that stabilizes  the time delay system for each delay in the interval $[0,\taumax)$. In general this is an unsolved problem, and results have been confined to obtaining upper and lower bounds for $\taumax$.  
In \cite{middleton2007achievable}  upper bounds for some simple systems are presented, but in general they are not tight. Methods for finding lower bounds based on different methods have been proposed, e.g, using robust control \cite{wang1994representation, huang2000robust}, integral quadratic constrains \cite{kao2007stability} (see also \cite{megretski1997system}), and analytic interpolation \cite{qi2014fundamental, qi2017fundamental}.

Our present paper builds on the approach in \cite{qi2014fundamental}, \cite{qi2017fundamental}, which formulates a sufficient condition for the maximum delay margin in terms of an interpolation problem with a real weight and obtains a lower bound using a rational approximation of the weight. In the present paper we instead reformulate the interpolation problem as an infinite dimensional analytic interpolation problem and solve it directly using techniques from function theory and complex analysis.
 This is related to work on discrete time systems in \cite{karlsson2010theinverse, karlsson2009degree}; methods that can also be used for control design and implementation.
In addition, by introducing a constant shift, we show that the lower bound can be further improved. In this short paper we concentrate on the delay margin itself and leave a deeper study of control implementation to a future paper. 

The outline of the paper is as follows. In Section~\ref{sec:delay_margin} we define the  delay margin problem and describe the results in \cite{middleton2007achievable}, \cite{qi2014fundamental}, \cite{qi2017fundamental}. In Section~\ref{sec:improving_lower_bound} we modify the approach of  \cite{qi2014fundamental}, \cite{qi2017fundamental} to obtain better lower bounds and provide an algorithm for this. This method is then improved in Section~\ref{sec:improved_method} by a simple shift of the corresponding complementary sensitivity function. Section~\ref{sec:example} is devoted to some numerical simulations. To facilitate comparison with the results in \cite{qi2017fundamental} we use some of the same systems as there. In Section~\ref{sec:control_implementation}  we provide a succinct discussion of control implementation, and in Section~\ref{sec:conclusions}  we discuss some possible future directions of research. 


\section{The delay margin problem}\label{sec:delay_margin}

Let $P(s)$ be the transfer function of a continuous-time, finite-dimensional, single-input-single-output LTI system, and 
consider the feedback control system depicted in Figure~\ref{fig:blockdiagram}. Here $e^{-\tau  s}$ is a delay, and $K(s)$ is  
a feedback controller in the class 
\[
\quotientfield := \left\{ \frac{N(s)}{D(s)} \; \Big| \; N,D \in \HinfC \text{ and } D(s) \not \equiv 0 \right\},
\]
where $\mC_+$ denotes the open right half plane, and $\HinfC$ denote the Hardy space of bounded analytic functions on $\mC_+$; see, e.g., \cite{foias1996robust}.
The basic problem in control theory is to find a $K(s)$ in this quotient field that stabilizes the closed loop system for a class of systems.

Let us first consider the standard problem without delay ($\tau=0$). The closed loop system is stable if 
\begin{equation}
\label{eq:stability}
1 + P(s)K(s) \neq 0 \quad \text{for all $s \in \bar{\mC}_+$},
\end{equation}
where $\bar{\mC}_+$ is the closed right half plane. This is equivalent to that the sensitivity function
\[
S(s): = (1 + P(s)K(s))^{-1}
\]
belongs to $\Hinf$, which in turn is equivalent to $T\in\Hinf$, where 
\[
T(s) := 1-S(s) = P(s)K(s) \big(1 + P(s)K(s) \big)^{-1}
\]
is the complementary sensitivity function \cite{doyle1992feedback}.  The feedback system is {\em internally stable\/} if, in addition, there is no pole-zero cancellation between $P$ and $K$ in $\bar{\mC}_+$ \cite[pp. 35-36]{doyle1992feedback}. Assuming for simplicity that the poles and zeros are distinct, this is equivalent to the interpolation conditions%
\footnote{If the poles and zeros are not distinct the interpolation conditions need to be imposed with multiplicity \cite{youla1974single}.}
\begin{subequations}\label{eq:interpolation}
\begin{align}
  &   T(p_j) = 1,\quad j = 1,\ldots, n , \\
  &  T(z_j)= 0,\quad j = 1,\ldots, m,
\end{align}
\end{subequations}
where $p_1, \ldots, p_n$  are the unstable poles and $z_1, \ldots, z_m$ the nonminimum phase zeros of $P$, respectively; see, e.g., \cite{youla1974single}, \cite[Chapters 2 and 7]{helton1998classical}. In the sequel we shall simply say that $K$ stabilizes $P$ when all these conditions are satisfied. 
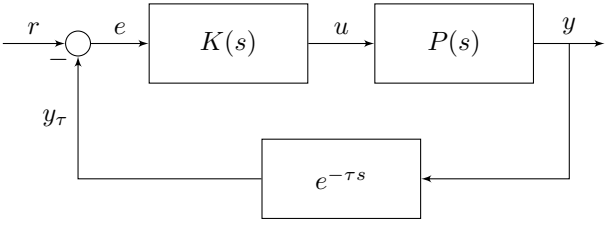
\begin{figure}[tb]
\tikzstyle{int}=[draw, minimum size=2em]
\tikzstyle{init} = [pin edge={to-,thin,black}]
\tikzstyle{block} = [draw, rectangle, 
    minimum height=3em, minimum width=6em]
\tikzstyle{sum} = [draw, circle, node distance=1cm]
\tikzstyle{input} = [coordinate]
\tikzstyle{output} = [coordinate]
\tikzstyle{pinstyle} = [pin edge={to-,thin,black}]

\begin{tikzpicture}[auto, node distance=2cm,>=latex']
    \node [input, name=input] {};
    \node [sum, right of=input] (sum) {};
    \node [block, right of=sum] (controller) {$K(s)$};
    \node [block, right of=controller, node distance=3cm] (system) {$P(s)$};

    \draw [->] (controller) -- node[name=u] {$u$} (system);
    \node [output, right of=system] (output) {};
    \node [block, below of=u] (delay) {$e^{-\tau s}$};

    \draw [draw,->] (input) -- node {$r$} (sum);
    \draw [->] (sum) -- node {$e$} (controller);
    \draw [->] (system) -- node [name=y] {$y$}(output);
    \draw [->] (y) |- (delay);
    \draw [->] (delay) -| node[pos=0.99] {$-$} 
        node [near end] {$y_\tau$} (sum);
\end{tikzpicture}
\caption{Block diagram representation of an LTI system with time delay.}
\label{fig:blockdiagram}
\end{figure}

If $K$ stabilizes $P$, by continuity it  also stabilizes $Pe^{-\tau s}$ for sufficiently small $\tau >0$. The question is how large $\tau$ can be while retaining internal stability. Following \cite{middleton2007achievable} we define the \emph{delay margin} for a given controller $K$ as
\begin{align*}
DM(P, K) & := \sup_{\tau \geq 0} \; \tau \\
  & \text{such that } K \text{ stabilizes } Pe^{-t s} \text{ for } t \in [0, \tau],
\end{align*}
and the \emph{maximum delay margin} for a plant $P$ as
\begin{align*}
\taumax = DM(P) & := \sup_{K \in \quotientfield} \; DM(P, K).
\end{align*}
This means that $\taumax$ is the largest value such that for any $\bar{\tau} < \taumax$ there exists a controller $K$ that stabilizes the plant $P$ for all $\tau$ in the  interval $[0, \bar{\tau}]$. 
If the plant $P$ is stable we trivially have $\taumax = \infty$, since $K \equiv 0$ stabilizes it, and thus we shall only consider unstable plants.

To determine $\taumax$ is in general a hitherto unsolved problem, but work has been done to obtain lower and upper bounds.

\subsection{Upper bounds for maximum delay margin}

In \cite{middleton2007achievable} it was shown that for any strictly proper real-rational plant $P$ with unstable poles in $re^{\pm i \theta}$, $r >0$ and $\theta \in [0, \pi/2]$, there is an upper bound $\bar{\tau}$ for $\taumax$ given by
\begin{equation}\label{eq:U_bounds}
\bar{\tau} = \frac{1}{r} \Big(\pi \sin(\theta) + 2\max\left\{\cos(\theta), \; \theta \sin(\theta) \right\} \Big)
\end{equation}
 \cite[Thm. 7, 9 and 11]{middleton2007achievable}.
Moreover, this upper bound is in fact shown to be tight in the special cases of either exactly one real unstable pole or exactly two conjugate unstable poles.
These results are the first that show that there is an upper bound on the achievable delay margin when using LTI controllers, and they describe a region for the delay where stabilization is not possible. However, the provided bounds of the maximum delay margin are in general not tight, and have lately also been improved upon in \cite{ju2016further, ju2018achievable}.

\subsection{Lower bounds for maximum delay margin}

To ensure stability we are in general more interested in a lower bound $\bar{\tau}\leq \taumax$. This problem is considered in the recent papers \cite{qi2014fundamental,qi2017fundamental}, where an approach based on analytic interpolation and rational approximations is taken. The starting point is that \eqref{eq:stability} can be written
\begin{equation}\label{eq:nec_and_suff_cond}
1+ T(s)(e^{-\tau s}-1)\neq 0 \quad \text{for  $s\in \bar{\mC}_+$},
\end{equation}
where $T$ is the complementary sensitivity function.  A sufficient condition for \eqref{eq:nec_and_suff_cond} to hold for all $\tau$ on an interval $[0,\bar{\tau}]$ is that
\begin{equation} \label{eq:suff_cond0}
\sup_{\tau\in [0, \bar{\tau}]} \; \inf_{\substack{T \in \Hinf \\ \text{subject to } \eqref{eq:interpolation}}} \|T(s)(e^{-\tau s}-1)\|_{\Hinf}<1.  
\end{equation}
Now, since $\sup  \, \inf \leq \inf \, \sup$, this condition holds whenever
\begin{equation}\label{eq:suff_cond}
\inf_{\substack{T \in \Hinf \\ \text{subject to } \eqref{eq:interpolation}}} \|T(i \omega) \phi_{\bar{\tau}}(\omega)  \|_{\Linf}<1,
\end{equation}
where
\begin{align}
\phi_{\bar{\tau}}(\omega) & =  \sup_{\tau\in [0, \bar{\tau}]}  | e^{-i\tau \omega}-1| \nonumber \\
& = \begin{cases}2\left|\sin(\frac{\bar{\tau} \omega}{2})\right| &\mbox{ for } |\omega \taumax| \le \pi\\
2 &\mbox{ for } |\omega \taumax| > \pi.
\end{cases} \label{eq:w_tau}
\end{align}
In \cite{qi2017fundamental} the function $\phi_{\bar{\tau}}$ is approximated by  the magnitude of  a rational function  $w_{\bar{\tau}}$ such that  $\phi_{\bar{\tau}}(\omega)\leq |w_{\bar{\tau}}(i\omega)|$ for all $\omega$. Using this approximation and the interpolation conditions on $T$ for internal stability the authors derive an algorithm for computing the largest $\bar{\tau}$ for which \eqref{eq:suff_cond} holds. This thus gives a lower bound for the maximum delay margin.

\section{Formulating  and solving \eqref{eq:suff_cond} using\\ analytic interpolation}\label{sec:improving_lower_bound}


In this section we will solve the problem \eqref{eq:suff_cond} directly using analytic interpolation without resorting to approximation of $\phi_{\bar{\tau}}(\omega)$ via rational functions. 
Continuing in the manner of \cite{qi2017fundamental} we note that  \eqref{eq:suff_cond}, the sufficient condition for the closed loop system to be internally stable for all $\tau\in [0, \bar{\tau}]$, holds if there exists a $T(s) \in \HinfC$ such that
\begin{equation}\label{eq:basic_interpolation_problem_1}
\| T(i\omega) \phi_{\bar{\tau}}(\omega) \|_{\Linf} \! <  1
\; \text{and}
\begin{cases} T(p_j)=1, \; j = 1, \ldots, n, \\ T(z_j)=0, \; j = 1, \ldots, m.
\end{cases} \!\!\!\!\!\!\!
\end{equation}
Next, we may replace $\phi_{\bar{\tau}}$ by the outer function  $W_{\bar{\tau}}\in \HinfC$ with the same magnitude as $\phi_{\bar{\tau}}$ on $i\mR$ \cite[p. 133]{hoffman1962banach}, and we arrive at the equivalent problem  
\begin{equation}\label{eq:basic_interpolation_problem_15}
\| T W_{\bar{\tau}} \|_{\Hinf} < 1
\; \text{and}
\begin{cases} T(p_j)=1, \; j = 1, \ldots, n, \\ T(z_j)=0, \; j = 1, \ldots, m,
\end{cases}
\end{equation}
where
\begin{equation}
\label{eq:outerrepr}
W_{\bar{\tau}}(s) = \exp \! \left[ \frac{1}{\pi} \int_{-\infty}^\infty \!\! \log \big( \phi_{\bar{\tau}}(\omega) \big) \frac{ \omega s + i}{\omega + is} \frac{1}{1+\omega^2} \,d\omega \right] \!\! .
\end{equation}
 Observing that $W_{\bar{\tau}}$ is outer, and setting $\tilde{T}:=TW_{\bar{\tau}}$, \eqref{eq:basic_interpolation_problem_15}
is seen to be equivalent to 
\begin{equation}\label{eq:basic_interpolation_problem_2}
\| \tilde T \|_{\Hinf} \! < 1
\; \text{and}
\begin{cases}
 \tilde T(p_j)=W_{\bar{\tau}}(p_j), &j = 1, \ldots, n, \\
 \tilde T(z_j)=0,  &j = 1, \ldots, m,
 \end{cases}
\end{equation}
and thus the only way the weight enters is through the values of the outer function $W_{\bar{\tau}}$ at the pole locations $p_j$ \cite[Section 4.C]{karlsson2010theinverse} (cf. \cite{karlsson2009degree}).
Since $W_{\bar{\tau}}$ is outer, no unstable poles or nonminimum-phase zeros have been added in $\mC_+$.

Hence we have reduced the problem to determining whether there exists a $\tilde T \in \Hinf$ such that \eqref{eq:basic_interpolation_problem_2} holds. The values $W_{\bar{\tau}}(p_j)$,  $ j = 1, \ldots, n$, can be computed from \eqref{eq:outerrepr} by numerical integration. Then setting 
\begin{subequations}\label{eq:pick_values}
\begin{align}
    &v := [p_1, \ldots, p_n, z_1, \ldots, z_m]   \\
    & w := [W_{\bar{\tau}}(p_1), \ldots, W_{\bar{\tau}}(p_n), 0, \ldots, 0] ,
\end{align}
\end{subequations}
 the interpolation problem \eqref{eq:basic_interpolation_problem_2} is solvable if and only if the corresponding Pick matrix
\begin{equation}\label{eq:pick_matrix}
\text{Pick}(v,w) := \begin{bmatrix}
\frac{1-w_j\bar w_k}{v_j+\bar{v}_k}
\end{bmatrix}_{j,k = 1}^{n+m}
\end{equation}
is positive definite;  see, e.g., \cite[pp. 151-152]{doyle1992feedback}. In case the poles and zeros are not distinct, \eqref{eq:pick_matrix} needs to be replaced by a more general criterion, e.g., using the input-to-state framework \cite{byrnes2001ageneralized,georgiou2002structure}
 as in \cite{blomqvist2005optimization}.

We have thus shown that for a given $\bar \tau$, the problem \eqref{eq:suff_cond} has a solution if and only if the Pick matrix \eqref{eq:pick_matrix} with interpolation values \eqref{eq:pick_values} is positive definite. Moreover, if \eqref{eq:suff_cond} has a solution for some $\bar{\tau}$ then clearly it has a solution for any smaller value, since $\phi_{\bar{\tau}}(\omega)$ is point-wise nondecreasing in $\bar{\tau}$. Therefore the optimal $\bar \tau$ 
can be computed using the bisection algorithm, iteratively testing feasibility of \eqref{eq:suff_cond}. The method is summarized in Algorithm~\ref{alg:first_method}.  
Note that by \eqref{eq:U_bounds} we have $2\pi/\max_j(|p_j|)\ge\tau_{\rm max}$, which gives a valid choice for the initial upper bound
in the bisection algorithm.

\renewcommand{\algorithmicrequire}{\textbf{Input:}}
\renewcommand{\algorithmicensure}{\textbf{Output:}}
\algsetup{indent=12pt}

\begin{algorithm}
\caption{Lower bound on maximum delay margin}
\label{alg:first_method}
\begin{algorithmic}[1]
\REQUIRE Unstable poles $p_j$,  $j = 1, \ldots, n$, and nonminimum phase zeros $z_j$,  $j = 1, \ldots, m$, of the plant $P$.

\STATE $\tau_- = 0.$
\STATE $\tau_+=2\pi/\max_j(|p_j|)$, 
\WHILE{$\tau_+ - \tau_- > \texttt{tol}$}
\STATE $\tau_{\rm mid} = (\tau_+ + \tau_-)/2$
\STATE Compute new interpolation values $W_{\tau_{\rm mid}}(p_j)$
\IF{Pick matrix \eqref{eq:pick_matrix} with values \eqref{eq:pick_values} is positive definite}
\STATE $\tau_- = \tau_{\rm mid}$
\ELSE
\STATE $\tau_+ = \tau_{\rm mid}$
\ENDIF
\ENDWHILE
\STATE $\bar{\tau}
 = \tau_-$
\ENSURE $\bar{\tau}
$, lower bound on maximum delay margin
\end{algorithmic}
\end{algorithm}

 The improvement of this method over that in \cite{qi2017fundamental} depends on how well the magnitude of the fifth-order approximation $w_{6\tau}(i\omega)$ used in \cite{qi2017fundamental} fits $\phi_{\bar{\tau}}(\omega)$ for $\omega\in \mR$. To illustrate this, the relative error for $\bar{\tau} = 1$ is shown in Figure~\ref{fig:approx}. In this particular case only a minor improvement in the lower bound is expected. 


However, our formulation of the problem allows for adding further constraints to the interpolation problem. This can be done in order to shape the sensitivity function, similarly to what has been done for discrete time systems in \cite{karlsson2010theinverse}. In the current setting this can be achieved by letting $\phi_{\rm design}$ be the modulus on the imaginary axis of the designed weight function and by  considering $\| T(i\omega)\phi_{\rm max}(\omega) \|_{\Linf} < 1$ in \eqref{eq:basic_interpolation_problem_1} instead, where $\phi_{\rm max}(\omega) = \max\{ \phi_{\bar\tau}(\omega), \phi_{\rm design}(\omega)\}$.

%

\begin{figure}[tb]
  \centering
  \includegraphics[width=0.4\textwidth]{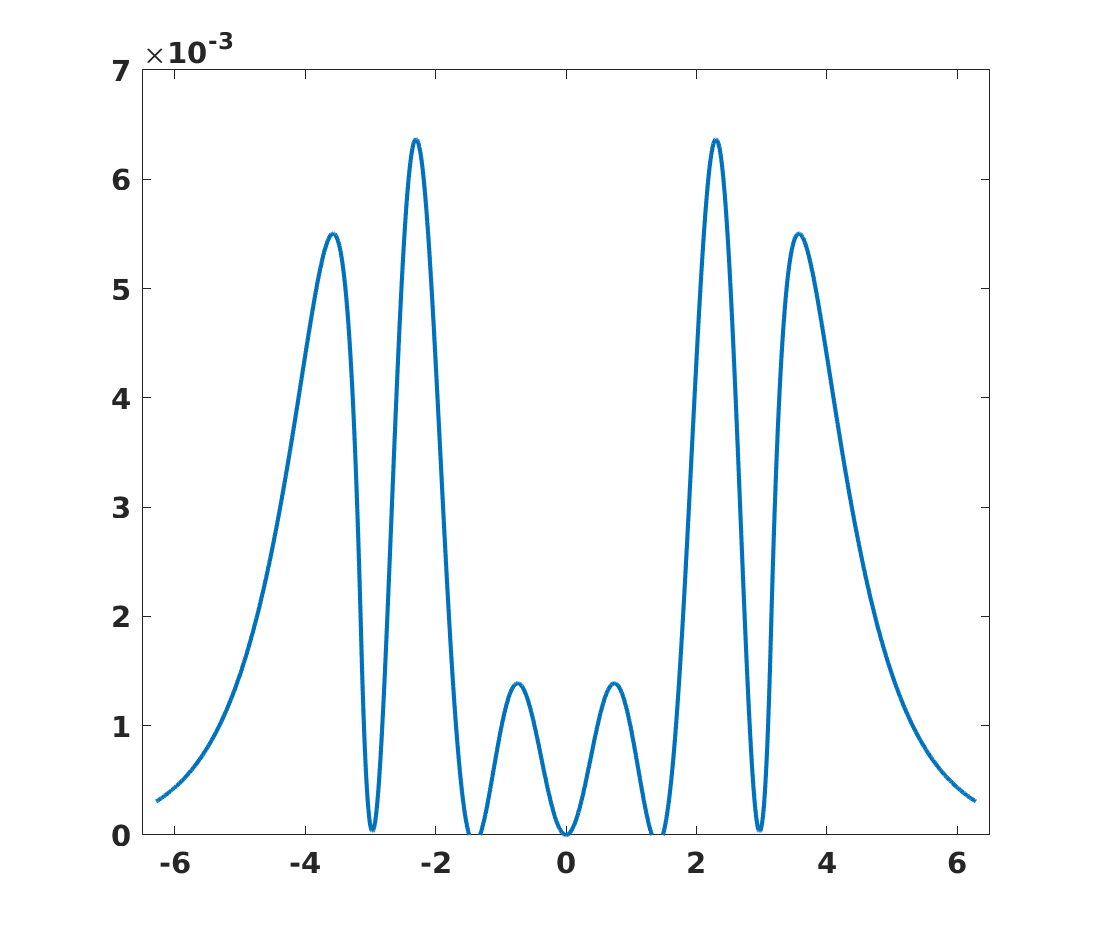}
  \caption{Relative error between $\phi_{\bar{\tau}}$ and the magnitude of fifth-order approximation $w_{6\tau}$ in \cite{qi2017fundamental}, for $\bar{\tau}
 = 1$. The relative error is given point-wise by $\big(|w_{6\tau}(i\omega)| - \phi_{\bar{\tau}}(\omega) \big) / \phi_{\bar{\tau}}(\omega)$.}
  \label{fig:approx}
\end{figure}

\section{Improving the lower bound \\using a constant shift}\label{sec:improved_method}

Consider the constraint $\| T(i\omega) \phi_{\bar{\tau}}(\omega) \|_{\Linf} < 1$ in \eqref{eq:basic_interpolation_problem_1}. For each $\omega$ the image of the complementary sensitivity function, $T(i\omega)$, is confined to a ball centered at the origin and with radius $| \phi_{\bar{\tau}}(\omega)^{-1}|$. However, choosing the center of the ball at the origin is quite arbitrary, and by instead carefully selecting the center elsewhere, we may improve the estimate of the lower bound. To this end, let $T = \hat{T} + w_0$ where $w_0 \in \mathbb{C}$. The condition \eqref{eq:nec_and_suff_cond} can then be written
\begin{equation}
\label{eq:nec_and_suff_cond_modified}
\hat{T}(s) \big( e^{-\tau s} - 1 \big) \neq -1 + w_0 - w_0e^{-\tau s}.
\end{equation}
Here the right hand side is an $\Hinf$ function, and it is nonzero in all of $\bar{\mC}_+$ if and only if $\real(w_0) < 1/2$, as can be seen from Lemma~\ref{lem:Rew0<1/2} in the appendix. Consequently, for $\real(w_0) < 1/2$, the inverse is an $\Hinf$ function and thus \eqref{eq:nec_and_suff_cond_modified} can be written as
\begin{equation}\label{eq:new_nec_and_suff_cond}
\hat{T}(s) \frac{e^{-\tau s} - 1}{1 - w_0 + w_0e^{-\tau s}} \neq -1.
\end{equation}
Hence we need modify the function $\absfuncwzero$ in Section~\ref{sec:improving_lower_bound} to read
\[
\absfuncwzero(\omega) := \sup_{\tau \in [0, \bar{\tau}
]} \left| \frac{e^{-\tau i\omega} - 1}{1 - w_0 + w_0e^{-\tau i\omega}} \right|,
\]
which reduces to \eqref{eq:w_tau} when $w_0=0$.
Then using the same argument as before, we see that 
\[
\| \hat{T}(i \omega) \absfuncwzero(\omega) \|_{\Linf}<1
\]
is  a sufficient condition for  \eqref{eq:new_nec_and_suff_cond} to hold. 

As shown in the appendix, $\absfuncwzero(\omega)$ can be determined in closed form, i.e., 
\begin{equation}\label{eq:absfuncwzero}
\absfuncwzero(\omega)^{-1} =
\begin{cases}
0.5 - \real(w_0), & \!\! \omega \geq \bar{\omega}_+, \\
\left|0.5 - i0.5\cot({\omega \bar{\tau}/2}) - w_0 \right|, & \!\! \bar{\omega}_+ \! > \! \omega \! > \! \bar{\omega}_-, \\
0.5 - \real(w_0), & \!\! \omega \leq \bar{\omega}_-,
\end{cases}
\end{equation}
where $\bar{\omega}_+$ and $\bar{\omega}_-$ are defined as follows: first define 
\[
\bar{\omega} := \frac{2}{\bar{\tau}
}\cot^{-1}(-2 \cdot \imag(w_0)),
\]
where we set $\cot^{-1}(0) = \pi/2$. Moreover, note that $\bar{\omega} \neq 0$ for any finite $w_0$. 
Next, define $\bar{\omega}_+$ and $\bar{\omega}_-$ by first setting $\bar{\omega}_+ = \bar{\omega}$ if $\bar{\omega} > 0$ or  $\bar{\omega}_- = \bar{\omega}$ if $\bar{\omega} < 0$ and then defining  the remaining variable via
\[
\bar{\omega}_+ = \bar{\omega}_- + 2\pi/\bar{\tau}.
\]

Following the procedure in Section~\ref{sec:improving_lower_bound} we define, via the representation \eqref{eq:outerrepr}, an outer function $\outerfuncwzero(s)$ with the property $|\outerfuncwzero(i\omega)|=\absfuncwzero(\omega)$ for all points on the imaginary axis. Consequently, we are left with the problem to find a $\hat T$ such that
\begin{equation*}
\|\hat T \outerfuncwzero \|_{\Hinf}<1
\; \text{and} \;
\begin{cases}
\hat T(p_j)=1-w_0,  &j = 1, \ldots, n, \\
\hat T(z_j)=-w_0,  &j = 1, \ldots, m,
\end{cases}
\end{equation*}
which, in turn,  is equivalent to
\begin{equation*}
\|\tilde T\|_{\Hinf} \!\! < 1
\; \text{and}
\begin{cases}
\! \tilde T(p_j)=(1-w_0)\outerfuncwzero(p_j),  &\!\! j = 1, \ldots, n, \\
\! \tilde T(z_j)=-w_0 \outerfuncwzero(z_j),  &\!\! j = 1, \ldots, m.
\end{cases}
\end{equation*}
In the same manner as in Section~\ref{sec:improving_lower_bound} we can then determine feasibility by checking whether the corresponding Pick matrix \eqref{eq:pick_matrix} is positive definite. A refined algorithm for computing a lower bound for the maximum delay margin is thus obtained by suitable changes in Algorithm~\ref{alg:first_method}. 

\section{Numerical example}\label{sec:example}
In this section we investigate the performance of the method proposed in Section~\ref{sec:improved_method} on some examples. To facilitate comparison with the results of \cite{qi2017fundamental} we consider the various SISO-systems given in \cite[Ex.1]{qi2017fundamental}.

\subsection{Systems with one unstable pole and one nonminimum phase zero}
We begin with the system \cite[Eq. (41)]{qi2017fundamental}, i.e.,
\begin{equation}\label{eq:chen_ex_1_2}
P(s)=\frac{s-z}{s-p},
\end{equation}
where $z,p > 0$. As in \cite{qi2017fundamental} we set $z=2$ and compute an estimate for the delay margin for different values of $p$ in the interval $[0.3, 4]$. Results are shown in Figure~\ref{fig:ex2}. From this we can see that with $w_0 = -10$ we get a considerable improvement over the bound in \cite{qi2017fundamental} in the region $p< z = 2$, and in this case we get close to the theoretical bound from \cite{middleton2007achievable} (which is tight in this region). However, with $w_0 = -10$ our method seems to perform worse than \cite{qi2017fundamental} in the region $p > z = 2$. On the other hand, in this region the value $w_0=0.35$ achieves some improvement. Note that the true stability margin is, to the best of our knowledge, still unknown in this region.


\begin{figure}[tb]
  \centering
  \includegraphics[width=0.475\textwidth]{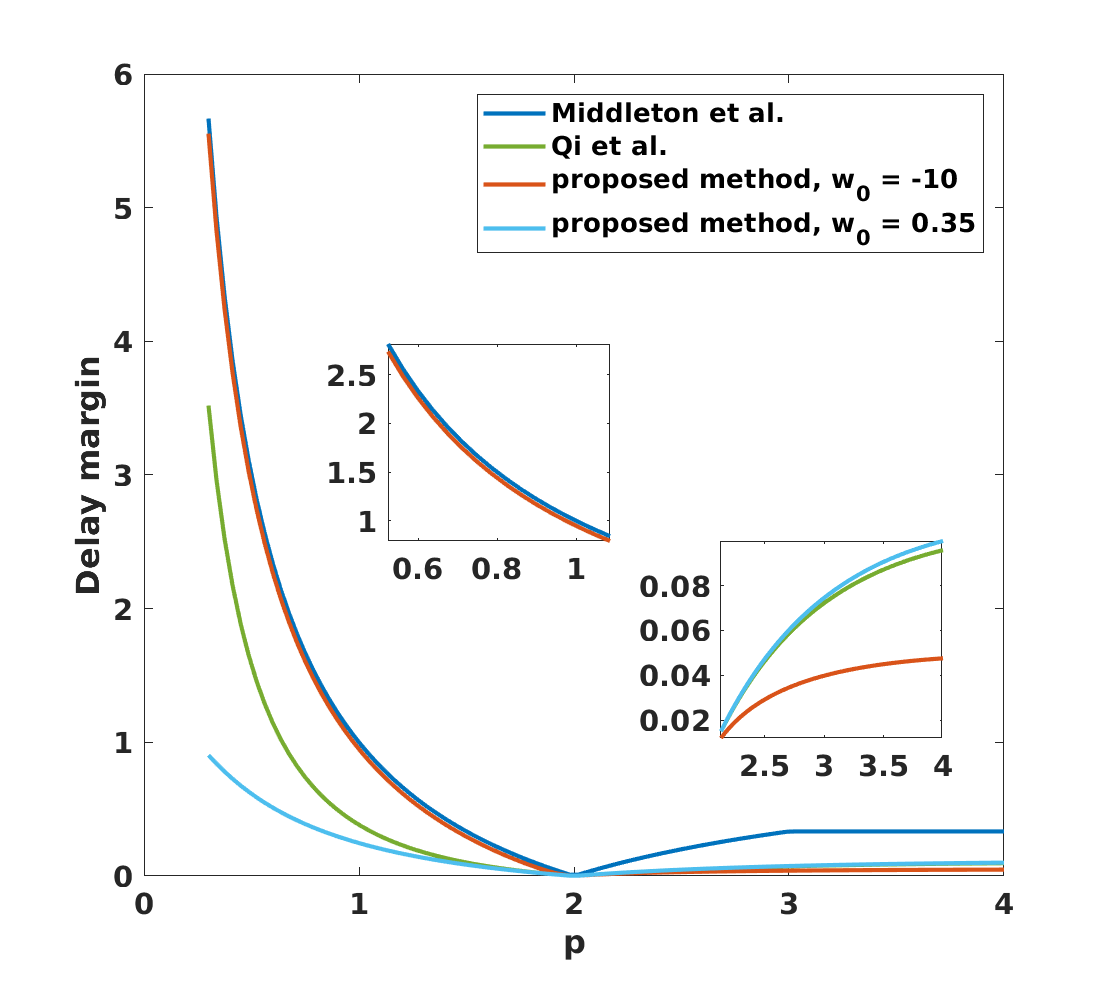}
  \caption{Results for the example in \eqref{eq:chen_ex_1_2}.}
  \label{fig:ex2}
\end{figure}

The system \cite[Eq. (42)]{qi2017fundamental}, given by
\begin{equation}\label{eq:chen_ex_1_3}
P(s)=0.1\frac{(0.1s-1)(s+0.1659)}{(s-0.1081)(s^2+0.2981s+0.06281)},
\end{equation}
has similar characteristics as the previous example, with one unstable pole ($p=0.1081$) and nonminimim phase zero ($z=10$). Also in this case our method gives a considerable improvement over \cite{qi2017fundamental} when $w_0$ is selected to be negative, and as $w_0$  tends to $-\infty$ our bound seems to approach the theoretical bound $2/0.1081 - 2/10 \approx 18.3$ from \cite{middleton2007achievable}; see Figure~\ref{fig:ex3}. 

\begin{figure}[tb]
  \centering
  \includegraphics[width=0.475\textwidth]{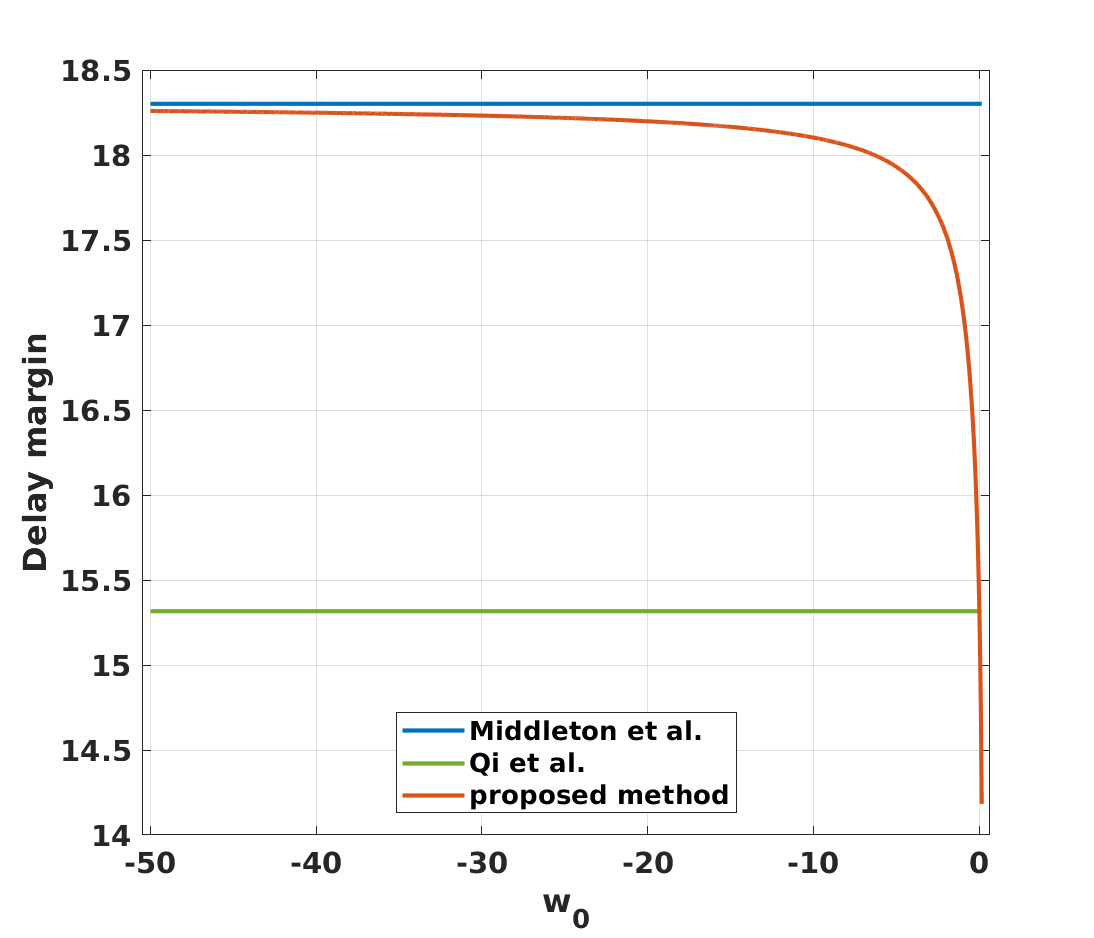}
  \caption{Results for the example in \eqref{eq:chen_ex_1_3}, with $w_0$ real. When $w_0$ goes to $-\infty$ we seem to get arbitrarily close to the result by Middleton et al. \cite{middleton2007achievable}, while for $w_0 > 0$ the bound deteriorate quickly.}
  \label{fig:ex3}
\end{figure}


\begin{figure*}[tb]
\begin{center}
\begin{subfigure}{.49\textwidth}
 \centering
 \includegraphics[width=\textwidth]{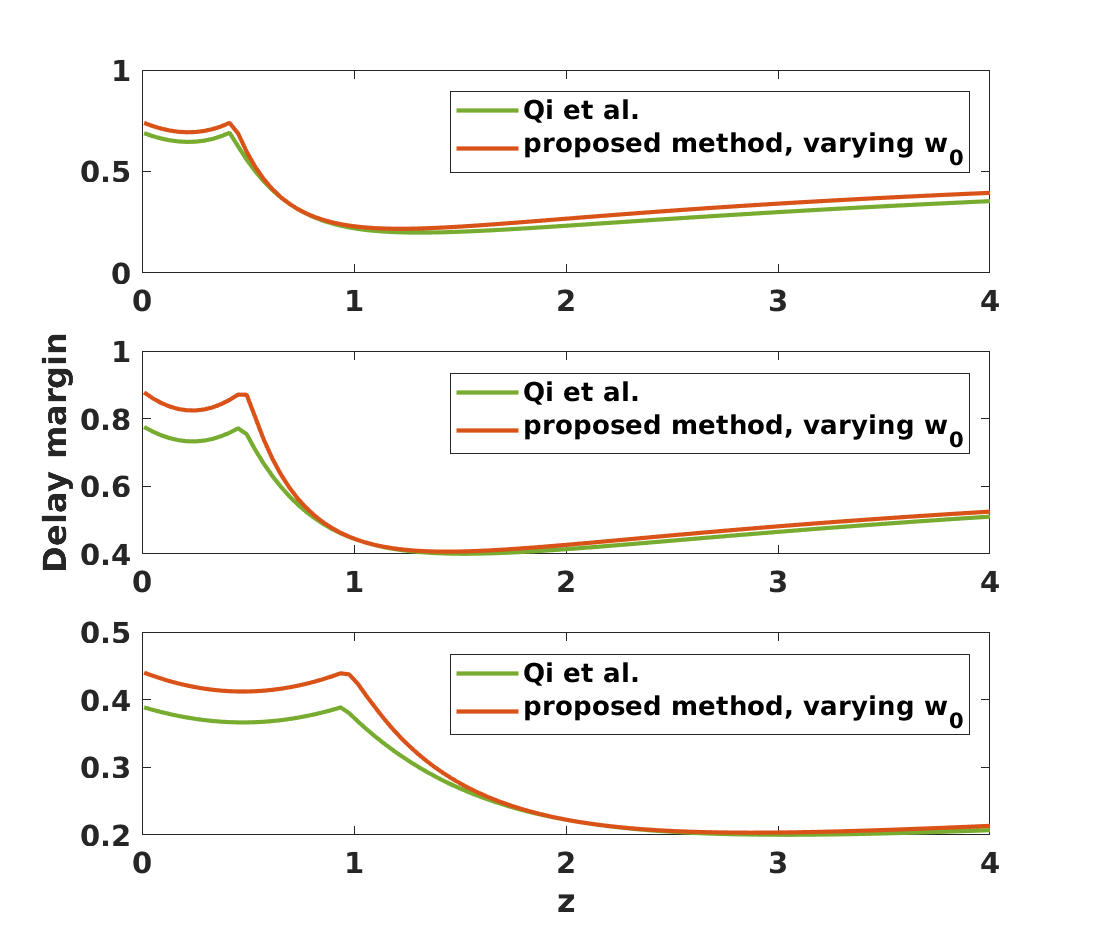}
 \subcaption{Estimates of the delay margin for the cases, from top to bottom,  $(r, \theta) = (1, \pi/4)$, $(r, \theta) = (1, \pi/3)$, and  $(r, \theta) = (2, \pi/3)$.}
 \label{subfig:ex4_delay_margin}
\end{subfigure}
\begin{subfigure}{.49\textwidth}
\centering
  \includegraphics[width=\textwidth]{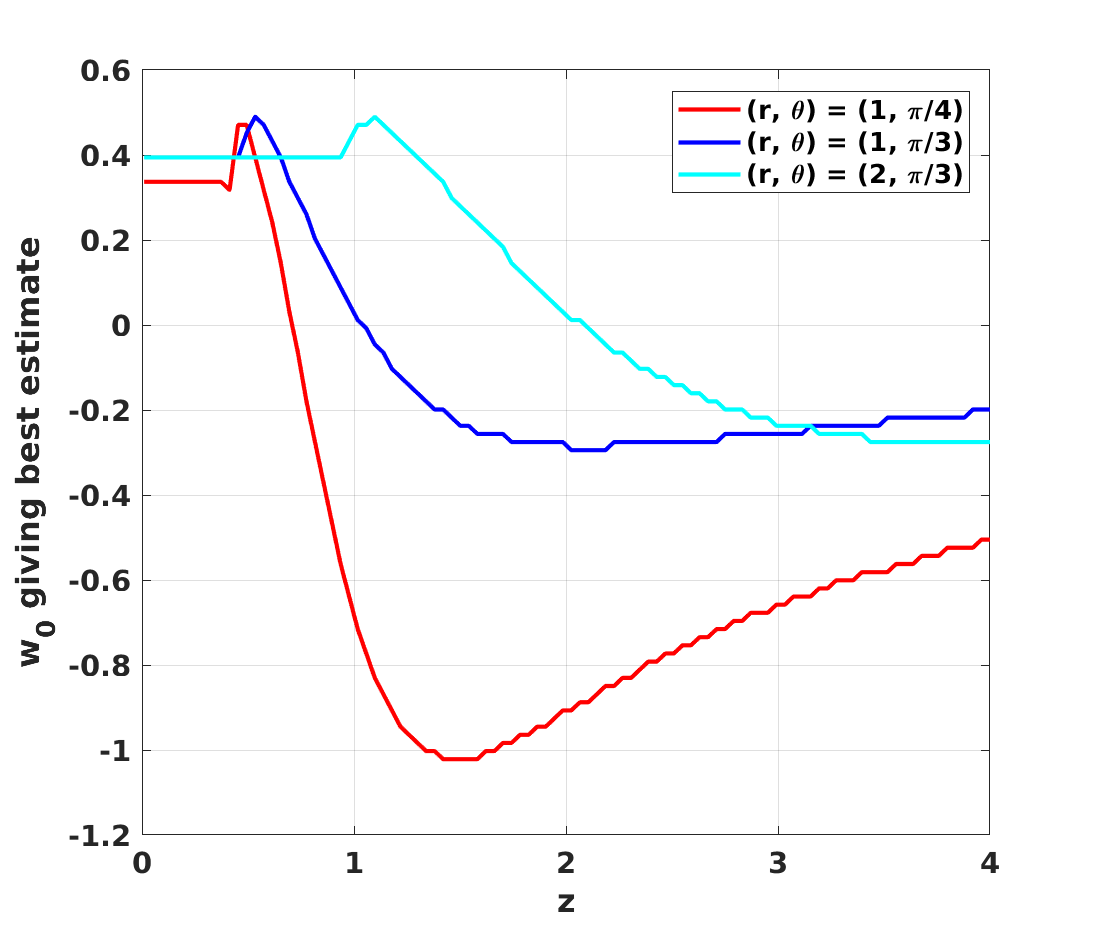}
 \subcaption{Best choice of $w_0$ as function of the zero position $z$.}
 \label{subfig:ex4_z_w0}
 \end{subfigure}
\caption{Reults for the example in \eqref{eq:chen_ex_1_4}.}\label{fig:ex4}
\end{center}
\end{figure*}

\subsection{System with two unstable real poles}
Next we consider the system \cite[Eq. (40)]{qi2017fundamental}, given by
\[
P(s)=\frac{1}{(s-p_1)(s-p_2)}.
\]
In this case $p_1$ is fixed to $0.2$, and the delay margin computed for different values of $p_2 \in [0.1, 3]$.
Then for values of $w_0 \in [-10, 0.5)$ only minor improvements over the result in \cite{qi2017fundamental} are achieved; for the corresponding optimal choice of $w_0$, the improvements are between $0.19\%$ and $2.9\%$ depending on $p_2$.


\subsection{System with conjugate pair of complex poles}
Finally we consider the system \cite[Eq. (45)]{qi2017fundamental}, which has a pair of unstable complex poles and a nonminimum phase zero. This system is given by 
\begin{equation}\label{eq:chen_ex_1_4}
P(s)=\frac{s-z}{(s-re^{i\theta})(s-re^{-i\theta})},
\end{equation}
and we compute an estimate of the delay margin for three fixed values of the pair $(r, \theta)$, namely for $(r, \theta) = (1, \pi/4)$, $(r, \theta) = (1, \pi/3)$, and $(r, \theta) = (2, \pi/3)$. Moreover, for these values of $(r, \theta)$ we vary $z$ in $[0.01, 4]$ and for each value of $z$ we investigate all values of $w_0 \in [-1.5, 0.5)$ (with steps $0.02$) to find the $w_0$ that maximizes the estimated delay margin. Results are shown in Figure~\ref{fig:ex4}, where Figure~\ref{subfig:ex4_delay_margin} shows the estimated delay margin and Figure~\ref{subfig:ex4_z_w0} shows the corresponding best value of $w_0$. The proposed method gives significantly improved bounds in some regions, for example when $\theta=\pi/3$ and $z$ is small compared to $r$.

\section{On the control implementation}\label{sec:control_implementation} 

There are certain problems with the implementation of the stabilizing controller that need attention. The complementary  sensitivity function is given by 
\begin{equation}
\label{eq:Ttilde2T}
T(s)=\tilde{T}(s) \outerfuncwzero (s)^{-1}+w_0.
\end{equation}
Indeed, since $\outerfuncwzero$ is outer, it is nonzero in $\mC_+$, and hence it can be inverted there. However, since $\outerfuncwzero(0)=0$, $T$ typically has a pole in $s = 0$, and therefore the closed loop system may not be stable (cf. \cite[p. 36]{doyle1992feedback}). 
This can be rectified  by replacing $\phi_{\bar{\tau}}$ by 
\[
\phi_{\bar{\tau}, {\varepsilon}}(\omega) =\max(\varepsilon, \phi_{\bar{\tau}}(\omega))
\]
for $\varepsilon > 0$.
This will give a stable system and, by continuity,  as $\varepsilon \to 0$ we can obtain a maximum delay margin estimate arbitrary close to $\bar{\tau}$. 

Selecting $\bar \tau$ to be the supremum for which \eqref{eq:suff_cond} holds gives rise to a singular Pick matrix \eqref{eq:pick_matrix} and a unique solution $\tilde T$ which is  a Blaschke product \cite[pp. 5-9]{garnett2007bounded}, so  $\| \tilde T \|_{\Hinf} =1$. Such a solution will not satisfy \eqref{eq:suff_cond} and thus may not have delay margin $\bar\tau$. However, for any $\bar \tau$ smaller than the supremum the Pick matrix is positive definite and the analytic interpolation problem (e.g., \eqref{eq:basic_interpolation_problem_2}) has infinitely many rational solutions \cite{byrnes2001ageneralized}, \cite{fanizza2007passivity}. We must now choose such a solution appropriately so that the stabilizing controller  
\begin{equation}
\label{eq:K}
K=P^{-1}(\tilde{T} +(1-w_0)W_{\bar{\tau}})^{-1}(\tilde{T} +w_0W_{\bar{\tau}}),
\end{equation}
is a rational function and thus can be implemented by a finite-dimensional system. Hence, unlike the approach in \cite{qi2014fundamental,qi2017fundamental}, an approximation may be needed to design the controller.
Again, methods similar to the ones presented in \cite{karlsson2010theinverse} can be used to obtain such an approximation, but details are left for a forthcoming paper.




\section{Conclusions and future directions}\label{sec:conclusions}
In this work we build on the approach in \cite{qi2014fundamental}, \cite{qi2017fundamental} for computing a lower bound for the maximum delay margin of a system. We introduce a parameter that can be tuned to improve the bounds, and in numerical examples we can in some cases come (arbitrarily) close to the true upper bound.
Subsequent work will focus on why this is the case, but also on how to tune the method and how to construct implementable controllers; the latter by following along the lines of \cite{byrnes2001ageneralized}, \cite{fanizza2007passivity}, \cite{karlsson2010theinverse}.


\section*{Acknowledgement}
We would like to thank Jie Chen for introducing us to the problem and for helpful discussions. We would also like to thank the referees for useful suggestions and comments.

\appendix

\subsection{A bound on $\Re(w_0)$ }

\begin{lemma}\label{lem:Rew0<1/2}
For $\tau > 0$ the function $h(s) = -1 + w_0 - w_0 e^{-\tau s}$ is nonzero in $\bar{\mC}_+$ if and only if $\real(w_0) < 1/2$.
\end{lemma}
\begin{proof}
Suppose $\tau >0$. If $\real(w_0)<0$, $h(s)$ is trivially nonzero for all $s\in\bar{\mC}_+$. Consequently we need only  consider the case $\real(w_0) \geq 0$. Then $\{w_0e^{-\tau s}\mid s\in\bar{\mC}_+\}= |w_0|\bar{\mathbb{D}}$, where $\bar{\mathbb{D}}$ is the closed unit disc $\{ s \in \mC \mid |s| \leq 1 \}$. Therefore $h(s)$ is nonzero if and only if $1 - w_0 \not \in |w_0| \bar{\mathbb{D}}$, which is true if and only if $|1 - w_0| > |w_0|$ which in turn is true if and only if  $\real(w_0) < 1/2$.
\end{proof}


\subsection{Computing $\absfuncwzero(\omega)$}
Since $\sup_x f(x) = \inf_x 1/f(x)$ we have that
\[
\absfuncwzero(\omega)^{-1} = \inf_{\tau \in [0, \bar{\tau}]} \left| w_0 - g(\omega,\tau) \right|,
\]
where $g(\omega,\tau):= (1-e^{-i\omega\tau})^{-1}$.
Introducing the set $$A_{\bar{\tau}}(\omega) := \big\{g(\omega,\tau) \mid \tau\in [0, \bar{\tau}
] \big\}$$ for each $\omega$,
\[
\absfuncwzero(\omega)^{-1} ={\rm dist}\Big(w_0 \;, \; A_{\bar{\tau}}(\omega) \Big),
\]
where ${\rm dist}(s_1, C) := \inf_{s_2 \in C} | s_1 - s_2|$
denotes the distance between a point and a set. Next we note that
\[
g(\omega,\tau)
= \frac{1}{2} - \frac{i}{2}\frac{\sin(\omega \tau)}{1-\cos(\omega \tau)}
= \frac{1}{2} - \frac{i}{2} \cot \left( \frac{\omega \tau}{2} \right),
\]
so $\Im \big( g(\omega,\tau) \big)$ is  a monotone increasing function of the product $\omega\tau$ in any interval $(0, 2\pi ) + k \cdot 2\pi$, $k \in \mathbb{Z}$.  Moreover,
\begin{displaymath}
g(\omega,\tau)-w_0=\left[\frac12 -\Re(w_0)\right] +i\left[ \Im \big(g(\omega,\tau)\big) -\Im(w_0)\right],
\end{displaymath}
where the real part is positive since we need $\real(w_0) < 1/2$ by Lemma~\ref{lem:Rew0<1/2}.  
Therefore $|g(\omega,\tau)-w_0|$ will take a minimum value when $\left|\Im \big(g(\omega, \tau)\big) -\Im(w_0)\right|$ is as small as possible. For a fixed $\omega \geq 0$, consider three cases. First, if $\omega \leq 2\pi/\bar{\tau}$ and if $\Im \big(g(\omega, \bar{\tau} )\big) \geq \Im(w_0)$, then, since $\Im \big(g(\omega, \tau)\big)$ is monotone increasing in $\tau$, $1/2-\Im(w_0)\in A_{\bar{\tau}}(\omega)$ and 
$|g(\omega,\tau)-w_0|\geq \frac12 - \Re(w_0)$, and hence 
\begin{align*}
{\rm dist}\Big(w_0 \;, \; A_{\bar{\tau}
}(\omega) \Big) = \frac{1}{2} - \Re(w_0).
\end{align*}
Second, if $\omega > 2\pi/\bar{\tau}$ the argument can be reduced to the above one by noticing that $\Im \big( g(\omega,\tau) \big)$ is $2\pi$-periodic in $\omega \tau$ and that $[0, 2\pi] \subset \{ \omega \tau \mid \tau \in [0, \bar{\tau}] \}$. Third,
if $\Im \big(g(\omega,\bar{\tau})\big) < \Im(w_0)$, then the minimum will be obtained for $\tau =\bar{\tau}$, so 
\[
{\rm dist}\Big(w_0 \;, \; A_{\bar{\tau}}(\omega) \Big) =  |w_0 - g(\omega,\bar\tau)|.
\]
In the same manner we obtain the analogous results for negative $\omega$.
Now define $\bar{\omega}_+\in (0,2\pi/\bar{\tau})$  to be the value of $\omega$ for which $\Im \big(g(\bar{\omega}_+,\bar{\tau})\big) = \Im(w_0)$, and let $\bar{\omega}_-\in (-2\pi/\bar{\tau},0)$ be the corresponding negative value. These are the frequencies at which $\absfuncwzero^{-1}$ changes form. Moreover, they can be computed by using $\bar{\omega}$ as in Section~\ref{sec:improved_method}.


\balance
\bibliographystyle{plain}

\bibliography{ref}

\end{document}